\documentclass[letterpaper,final,10pt]{amsart}
\usepackage[utf8x]{inputenc}

\usepackage{amsmath}%
\usepackage{amsfonts}%
\usepackage{amssymb}%
\usepackage{graphicx}
\usepackage{functan}
\usepackage{mathrsfs}
\usepackage[obeyspaces]{url}
\usepackage{braket}
\usepackage[all,cmtip,curve,knot]{xy}
\usepackage{tikz}
\usetikzlibrary{arrows}
\tikzstyle{block}=[draw opacity=0.7,line width=1.4cm]
\usepackage{algorithm}
\usepackage{algorithmicx}
\usepackage{algpseudocode}
\usepackage{mathbbol}
\newtheorem{theorem}{Theorem}

\newtheorem{definition}{Definition}
\newtheorem{example}{Example}

\newtheorem{lemma}{Lemma}

\newtheorem{remark}{Remark}

\numberwithin{equation}{section}
\numberwithin{theorem}{section}
\numberwithin{lemma}{section}
\numberwithin{corollary}{section}
\numberwithin{definition}{section}
\numberwithin{example}{section}
\numberwithin{remark}{section}
\numberwithin{property}{section}
\numberwithin{proposition}{section}

\newcommand{\NS}[3][1]{#2_#1,\ldots,#2_#3}

\newcommand{\lift}[3][\pi]{\xymatrix{#2 \ar@{~>}[r]_{#1} & #3}}

\newcommand{\Minf}{\mathcal{M}_\infty}

\newcommand{\Li}[2][]{\mathcal{L}_{#1}(#2)}

\newcommand{\Ad}[2]{\mathrm{Ad}[#1](#2)}
\newcommand{\diag}[2][]{\mathrm{diag}_{#1}\left[ #2 \right]}
\newcommand{\N}[1]{\mathcal{N}(#1)}
\newcommand{\U}[1]{\mathbb{U}(#1)}

\newcommand{\I}{\mathbb{1}}

\newcommand{\RR}{\mathbb{R}}
\newcommand{\CC}{\mathbb{C}}
\newcommand{\ZZ}{\mathbb{Z}}

\newcommand{\Rep}[2]{\mathrm{Rep}(#1,#2)}

\newcommand{\TT}[1][1]{{\mathbb{T}^{#1}}}

\newcommand{\disk}[1][2]{{\mathbb{D}^{#1}}}

\newcommand{\Cliff}[1]{\mathrm{Cliff}(#1)}
\newcommand{\dist}[2]{\mathbf{\Delta}(#1,#2)}

\begin{document}
\title[Matrix Words]{Local Deformation of Matrix Words}
\author{Fredy Vides}
\address
{Department of Mathematics and Statistics \newline
\indent The University of New Mexico, Albuquerque, NM 87131, USA.}
\email{vides@math.unm.edu}


\date{\today}

\begin{abstract}
In this document we study some local deformation properties of matrix representations of the universal C$^*$-algebras 
denoted by $\mathbb{I}^{m}_\varepsilon[p_1,\ldots,p_m]$ and $\mathbb{S}^{m-1}_\varepsilon[p_1,\ldots,p_m]$, and that we call 
{\bf Semi-Soft Cubes} and {\bf Semi-Soft Spheres} respectively.

We will use some C$^*$-algebraic technology to study the local deformation properties of matrix words in particular representations of Semi-Soft 
Cubes and Spheres, we will then use these results to study the local deformation properties of generic matrix equations on words.

Some geometrical aspects of the local deformation of matrix words will be addressed, and some connections with matrix numerical analysis and computational 
physics will be outlined as well.
\end{abstract}

\maketitle

\section{Introduction}
\label{intro}

In this document we study some local deformation properties of matrix representations of the universal C$^*$-algebras 
$\mathbb{I}^{m}_\varepsilon[p_1,\ldots,p_m]$ and $\mathbb{S}^{m-1}_\varepsilon[p_1,\ldots,p_m]$ that we call in this document {\bf Semi-Soft Cubes} and 
{\bf Semi-Soft Spheres} respectively, and that will be defined in \S\ref{Applications}.

We will use some C$^*$-algebraic technology to study the local deformation properties of matrix words in particular representations of Semi-Soft 
Cubes and Spheres, we will then use these results to study the local deformation properties of generic matrix equations on words.

\subsubsection{C*-Algebras and Topologically Controlled Linear Algebra.}

%

 Given $\delta>0$, a function $\varepsilon:\RR\to\RR^+_0$, a finite set of functions $F \subseteq C(\TT,\disk)$ and two unitary matrices $u,v\in M_n$ such that 
 $\|uv-vu\|\leq \delta$, we say that the set 
$F$ is $\delta$-controlled by $\mathrm{Ad}[v]$ if the diagram,
\begin{equation}
\xymatrix{
C^*(u,v) & C^*(u) \ar[l] \ar[d]_{\mathrm{Ad}[v]}  & \{u\} \ar[d]_{\mathrm{Ad}[v]} \ar[l]^{\imath} \ar[rd]^{f}_{\approx_{\varepsilon(\delta)}} &  \\
& C^*(vuv^*) \ar[lu]  & \{vuv^*\} \ar[l]^{\imath} \ar[r]_{f} & \N{n}(\disk)
}
\label{controlled_stuff}
\end{equation}
commutes up to an error $\varepsilon(\delta)$ for each $f\in F$.

In \cite{Vides_homotopies} we found some connections between the previously described 
C$^*$-algebras and {\bf \em topologically controlled linear algebra} (in the sense of Freedman and Press 
\cite{CLA_Freedman}), which can be roughly described as the study of the relations between matrix sets and 
smooth manifolds that is performed by implementing techniques from geometric topology in the study of matrix approximation problems. The connections 
we found can be outlined with the help of diagram \ref{controlled_stuff} together with matrix embeddings of the form 
\begin{equation}
 C^*(\NS{U}{N}) \hookrightarrow C^*(U,V),
\label{representation_diagrams_1}
 \end{equation}
for some $U,V,\NS{U}{N}\in \U{n}$ such that, $\|UV-VU\|\leq \varepsilon$, $U_j\in C^*(U,V)$ and $U_jU_k=U_kU_j$, $1\leq j,k\leq N$. By modifying 
\ref{representation_diagrams_1} appropriately, one can also obtain particular matrix embeddings of the form 
\begin{equation}
 C^*(\NS{H}{{N}},\NS{K}{N}) \hookrightarrow C^*(\hat{H}_1,\hat{H}_2,\hat{K}_1,\hat{K}_2),
\label{representation_diagrams_2}
 \end{equation}
where $U=\hat{H}_1+i\hat{K}_1$, $V=\hat{H}_2+i\hat{K}_2$ and $U_j=H_j+iK_j$, $1\leq j\leq N$ are the cartesian decompositions of the unitary matrices 
in \ref{representation_diagrams_1}.

\subsubsection{Local Matrix Homotopies and Matrix Words}  \label{geometry}Given $\delta>0$, a function $\varepsilon:\RR\to \RR^+_0$ and two matrices $x,y$ in a set $S \subseteq M_n$ such that 
$\|x-y\|\leq \delta$, by a {\bf \em $\varepsilon(\delta)$-local matrix homotopy} between $x$ and $y$, we mean a 
matrix path $X\in C([0,1],M_n)$ such that $X_0=x$, $X_1=y$, $X_t\in S$ and $\|X_t-y\|\leq \varepsilon(\delta)$ for each $t\in [0,1]$.

\begin{definition}[Local matrix deformations $x \rightsquigarrow_{\varepsilon} y$]
 Given two matrices $x,y\in M_n$ we write $x\rightsquigarrow_{\varepsilon} y$ if there is $\varepsilon$-local matrix homotopy 
 $X\in C([0,1],M_n)$ between $x$ and $y$.
\end{definition}

As a consequence of some of the results reported in \S\ref{main_results} and the constructive nature of their proofs, we have learned about the convenience 
of what we call {\bf joint spectral clustering} when it comes to the computation of local deformations of $m$-tuples of normal matrices.

\begin{example}
For an example of spectral clustering, lets us consider two hermitian matrices $X,Y\in \CC^{400\times 400}$ such that 
$\|[X,Y]\|=O(1\times 10^{-3})$, let us now form the matrix $A=X+iY$ that we can use to compute the {\bf joint pseudospectrum} 
$\Lambda_\varepsilon(X,Y)=\Lambda_\varepsilon(A)$ of 
$X$ and $Y$ in the sense of Loring \cite{Loring_pseudospectra} that is shown in Figure \ref{picture_1}, we can also use the Ritz values of $A$ to compute an approximate 
(Chebyshev) minimal polynomial $p_{\delta,A}$ for $A$. The corresponding approximate polynomial lemniscates of $p_{\delta,A}$ together with a searching/interpolating 
interpolating Chebyshev type grid ares shown in Figures \ref{picture_2} and \ref{picture_3}.

\begin{figure}[!htb]
\centering
 \includegraphics[scale=0.34]{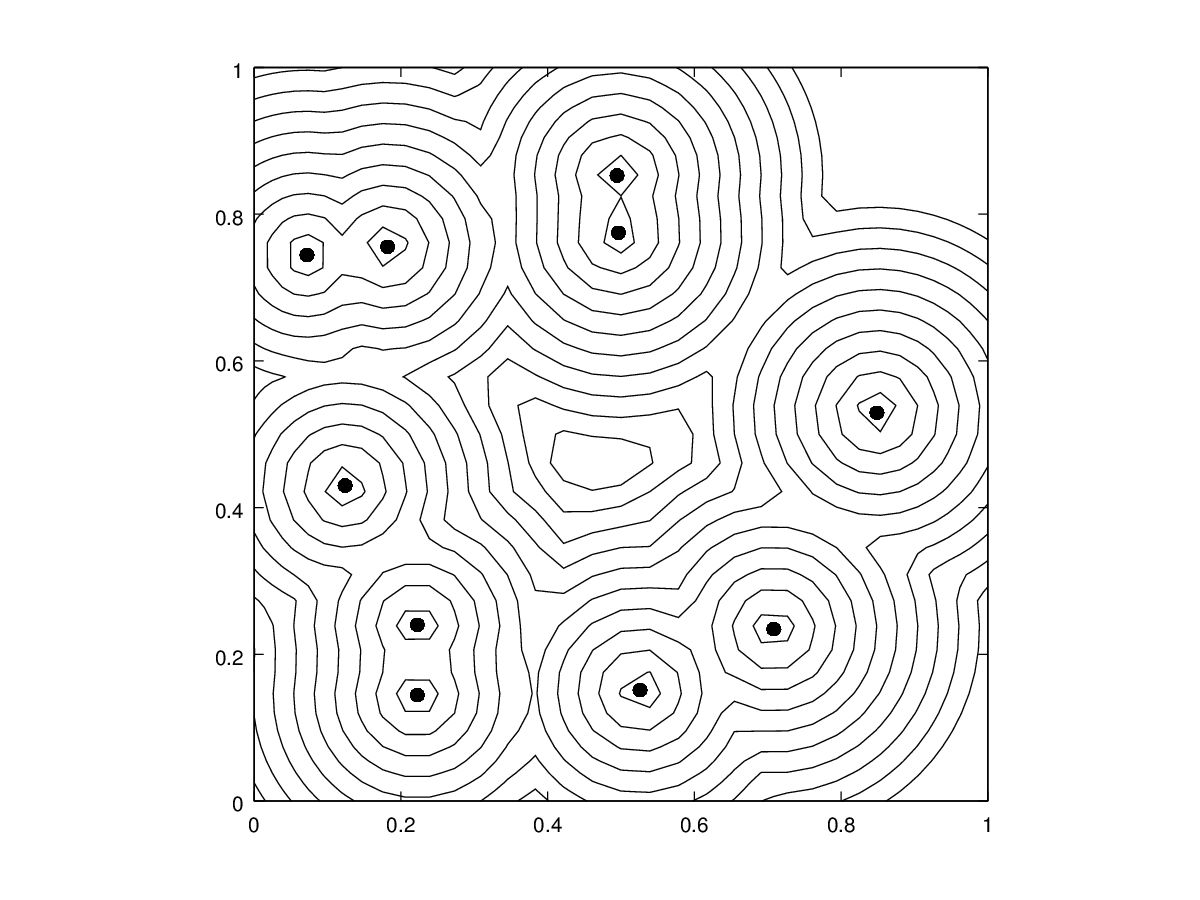}
 \caption{A section of the $\varepsilon$-Pseudospectrum $\Lambda_{\varepsilon}(A)$ of $A\in \CC^{400\times 400}$.}
 \label{picture_1}
 \end{figure}
  \begin{figure}[!htb]
 \centering
\includegraphics[scale=0.34]{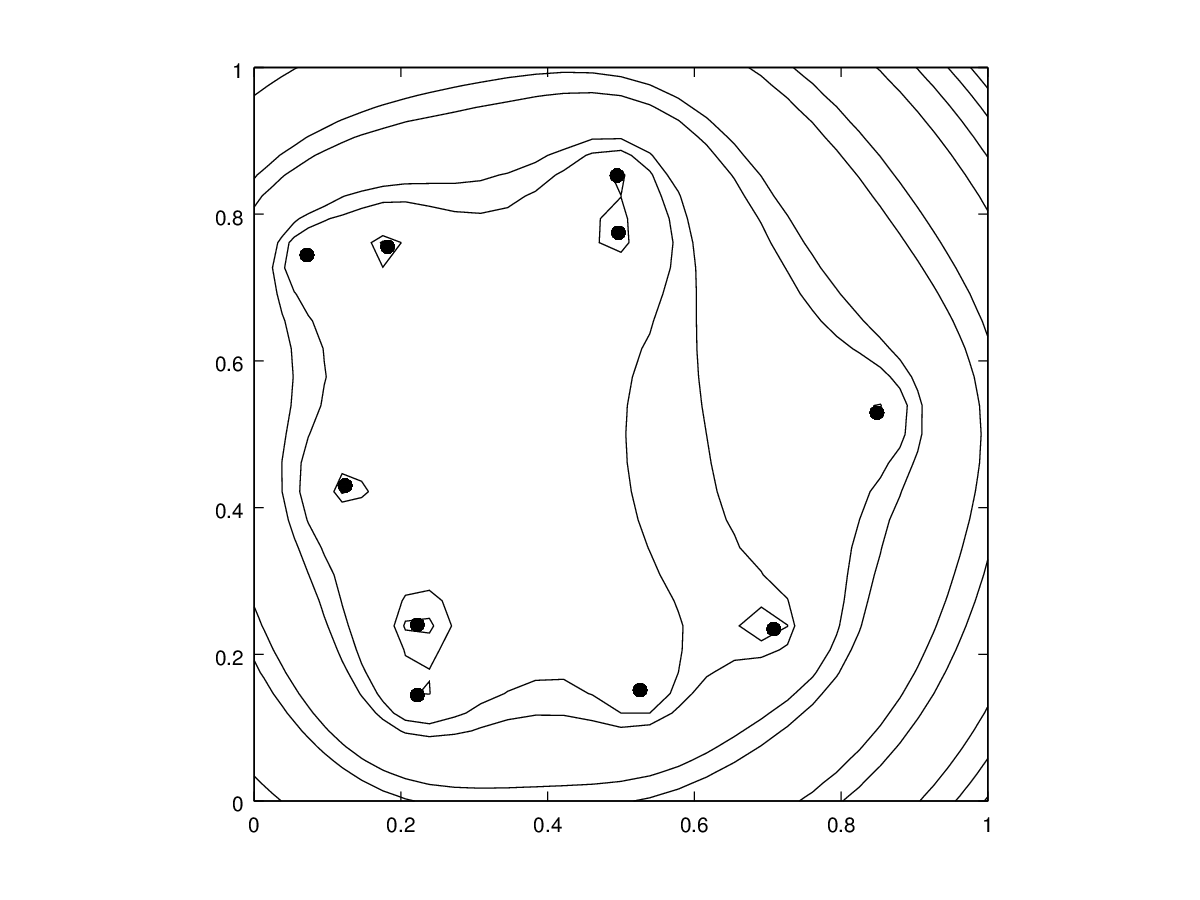}
\caption{Approximate polynomial lemniscates for an approximate minimal polynomial $p_{A,\delta}(z)$ for $A\in \CC^{400\times 400}$, 
where $\deg(p_{A,\delta}(z))= 10$ and $\|p_{A,\delta}(z)\|=O(1\times 10^{-3})$.}
\label{picture_2}
 \end{figure}
 \begin{figure}[!htp]
\centering
\includegraphics[scale=0.34]{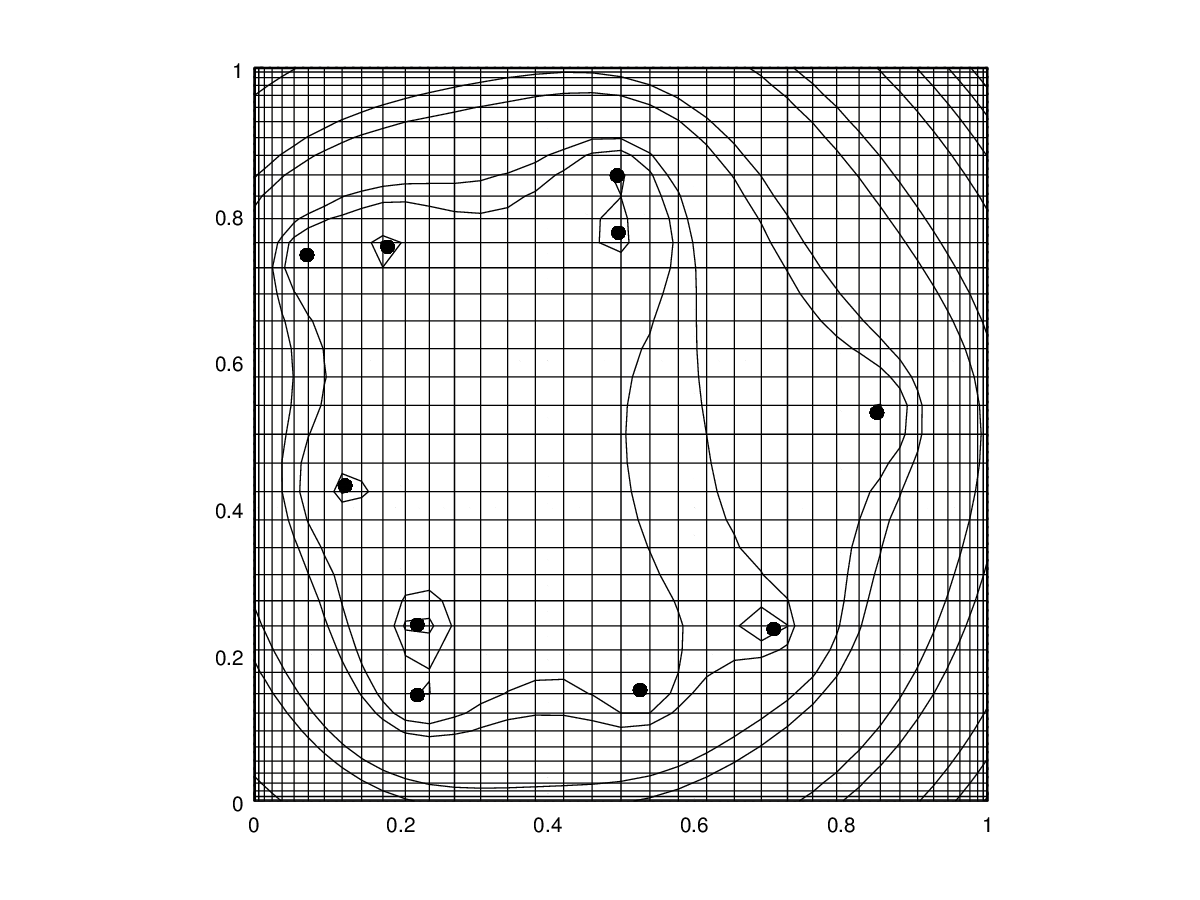}
\caption{Approximate polynomial lemniscates and interpolating Chebyshev type grid for an approximate minimal polynomial $p_{A,\delta}(z)$ for 
$A\in \CC^{400\times 400}$, where $\deg(p_{A,\delta}(z))= 10$ and $\|p_{A,\delta}(z)\|=O(1\times 10^{-3})$.}
\label{picture_3}
  \end{figure}
 
\end{example}

\subsubsection{Artificial joint spectral clustering} \label{clustering} 
Given $\varepsilon>0$, a matrix $A\in\CC^{N\times N}$, a region $\Omega^2\subseteq \CC\simeq \RR^2$ and a finite set of points $\tilde{\Omega}^2\subseteq \Omega^2$, by 
{\bf Pseudospectral scanning} we mean a triple $(U_{\varepsilon}(A,\tilde{\Omega}^2),\Lambda_\varepsilon(A,\tilde{\Omega}^2),V_{\varepsilon}(A,\tilde{\Omega}^2))$ 
formed by a set $\Lambda_\varepsilon(A,\tilde{\Omega}^2):=\{\tilde{\sigma}_{\varepsilon,1},\ldots,\tilde{\sigma}_{\varepsilon,s}\}\subseteq \tilde{\Omega}^2$ together with two  
sets of matrices $U_{\varepsilon}(A,\tilde{\Omega}^2):=\{\tilde{U}_1,\ldots,\tilde{U}_s\}$ and 
$V_{\varepsilon}(A,\tilde{\Omega}^2):=\{\tilde{V}_1,\ldots,\tilde{V}_s\}$, such that 
\begin{equation}
 \|\tilde{U}_jA\tilde{V}_j-\tilde{\sigma}_{\varepsilon,j}\tilde{U}_j\tilde{V}_j\|\leq \varepsilon,
 \label{Pseudospectral_scanning_constraint}
\end{equation}
for each $1\leq j\leq s$.

Given $0<\delta\leq \varepsilon$, a polynomial matrix function $f\in \CC[x,y]$ and a pair of $\delta$-commuting 
hermitian matrices $X,Y$ in $M_n$ such that $\|f(X,Y)\|\leq\varepsilon$, using perturbation theory one can combine Friis-Rørdam techniques (presented 
\cite{Rordam_Lin_Thm} to derive an alternative proof of Lin's theorem), together with the Pseudospectral scanning techniques presented by T. A. Loring 
in \cite{Loring_pseudospectra} to obtain an approximate interpolating matrix polynomial (in the sense of \cite{Matrix_Poly_Dennis}) 
$\tilde{f}(x,y)$ for $f(x,y)$, such that there are two commuting hermitian matrices $\tilde{X},\tilde{Y}\in M_n$ that satisfy the constraints 
$\tilde{f}(\tilde{X})=[\tilde{X},\tilde{Y}]=0$, 
$\|f(\tilde{X},\tilde{Y})\|\leq \varepsilon$ and $\max\{\|X-\tilde{X}\|,\|Y-\tilde{Y}\|\}\leq \varepsilon$.

\begin{figure}[!htb]
\centering
 \includegraphics[scale=0.34]{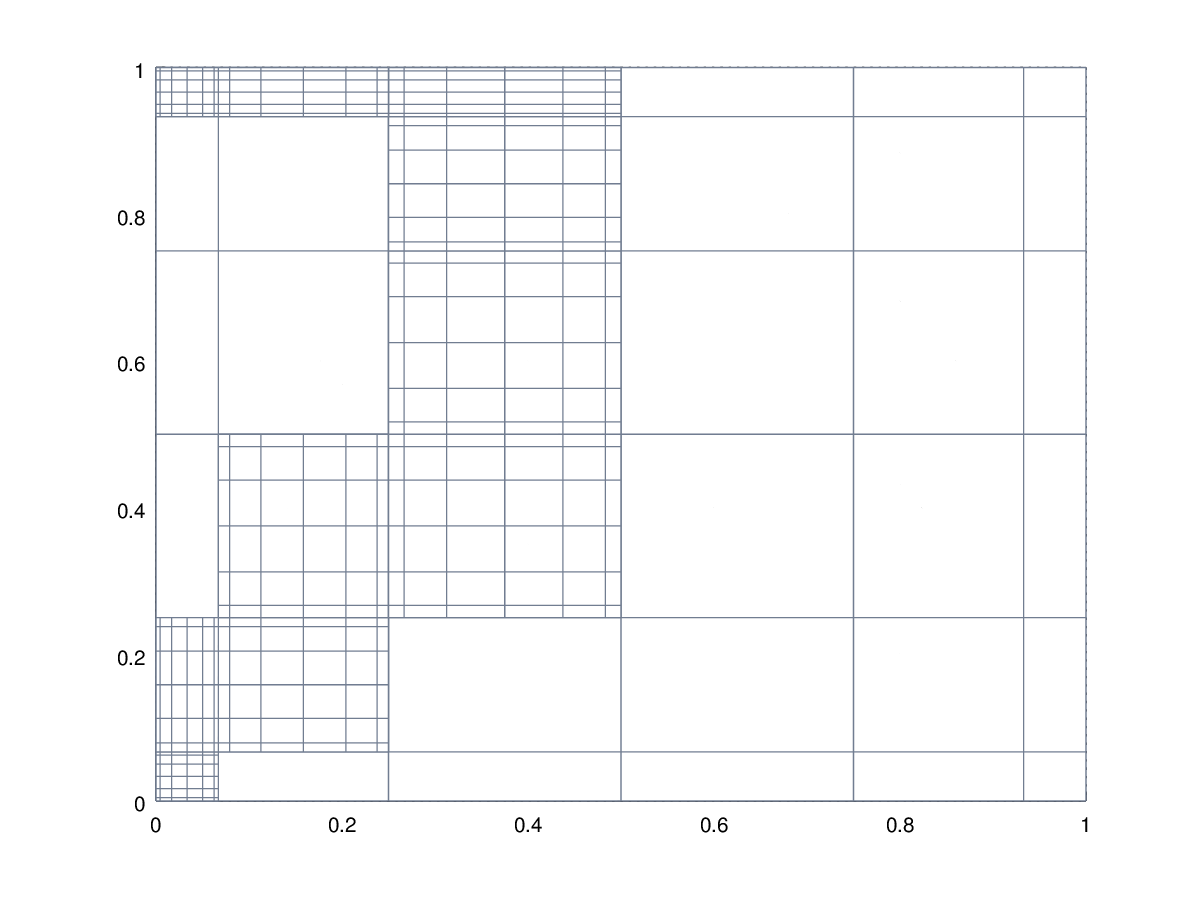}
 \caption{A locally refined interpolating grid for $[0,1]^2$.}
 \label{picture_4}
 \end{figure}

By combining matrix dilation techniques and joint Pseudospectral scanning of the estimated/preprocessed 
joint-Pseudospectral region in $\Omega_{\varepsilon}^m\subseteq \mathbb{R}^m$ of an $m$-tuple of almost commuting hermitian matrices
$X_1,\ldots,X_m$ in $M_n$, one can extend the interpolation method described in the previous paragraph to $\mathbb{R}^m$, the idea 
is essentially that by the results obtained by Loring and Exel in \cite{ExelLoringInvariants}, one can make (as long it is allowed by the nature of 
the matrix equations under study) {\em "surgical Pseudospectral cuts"} in $\tilde{\Omega}^m$ to throw away some {\em "bad spectral projectors"} 
with their corresponding invariant 
subspaces and joint-Pseudoeigenvalues in $\tilde{\Omega}^m$, in such a way that the geometry of the {\bf \em leftover} subregion of $\tilde{\Omega}^m$ 
is {\bf "nice enough"} to kill the Bott index (in the sense of \cite{Loring_assymptotic_K_theory}) of the corresponding 
{\em "$\varepsilon$-nearby/projected"} $m$-tuple $\tilde{X}_1,\ldots,\tilde{X}_m$ in $M_n$ corresponding
to the original matrices $X_1,\ldots,X_m \in M_n$.

By the results presented in \S\ref{main_results}, one can perform the Loring's Pseudospectral scanning in a pairwise fashion. Under the pairwise joint 
Pseudospectral scanning consideration, one can use grids like the one illustrated in Figure \ref{picture_3} to perform the Pseudospectral scanning. Building on 
the procedures presented in \S\ref{main_results}, the extension of Loring's Pseudospectral scanning 
techniques to Chebyshev type grids like the one presented in Figure \ref{picture_4}, by implementing and modifying some standard computational techniques for 
local refinement and multigrid pre- and pot-processing, seems promising in order to develope some computational tools for the study of problems in 
Graphene Nanotechnology (in the sense of \cite{NanoMesh1} and \cite{NanoMesh2}).

In this document we build on the techniques developed in \cite{Vides_dissertation} and \cite{Vides_homotopies}, to study the analytic local 
connectivity properties of particular representations in $\Minf$ of C$^*$-algebras generated by universal semialgebraic normal contractions, under 
commutativity preserving and semialgebraic constraints.

Some geometrical aspects of the local deformation of matrix words will be addressed in \S\ref{covering_varieties}, in particular we consider 
the {\bf curved} nature of local homotopies that preserve commutation relations. The main results will be presented in 
\S\ref{Applications}.

\section{Preliminaries and Notation}
\label{notation}


\begin{definition}[Semialgebraic Matrix Varieties]
\label{matrix_variety}
Given $J\in \ZZ^+$, a system 
of $J$ polynomials $\NS{p}{J}\in \Pi_{\braket{N}}=\mathbb{C}\Braket{\NS{x}{N}}$ in $N$ NC-variables $\NS{x}{N}\in \Pi_{\braket{N}}$ 
and a real number $\varepsilon\geq 0$, a particular matrix representation of the 
noncommutative semialgebraic set $\mathcal{Z}_{\varepsilon,n}(\NS{p}{J})$ 
described by 
\begin{equation}
 \mathcal{Z}_{\varepsilon,n}(\NS{p}{J}):=\Set{\NS{X}{N}\in M_{n} | \|p_j(\NS{X}{N})\|\leq \varepsilon, 1\leq j\leq J},
\end{equation}
will be called a {\bf $\varepsilon,n$-semialgebraic matrix variety} ($\varepsilon,n$-SMV), if $\varepsilon=0$ we can refer to the set as 
a {\bf matrix variety}. 
\end{definition}

Given any two matrices $X,Y\in M_n$ we will write $[X,Y]$ and $\mathrm{Ad}[X](Y)$ to denote the operations 
$[X,Y]:=XY-YX$ and $\mathrm{Ad}[X](Y):=XYX^*$.

Given a compact set $\mathbb{X}\subset \mathbb{C}$ and a subset $S\subseteq M_n$, let us dote by $S(\mathbb{X})$ the set 
$S(\mathbb{X}):=\{X\in S|\sigma(S)\subseteq \mathbb{X}\}$, in particular we will write $\N{n}(\disk)$ to denote the set 
normal contractions in $M_n$.

\begin{example} Given any integer $n\geq 1$, let us set $\mathbf{N}:=\diag{n,n-1,\ldots,1}$, we will have that the 
set $Z_\mathbf{N}:=\{X\in M_n|[\mathbf{N},X]=0\}$ is a matrix variety. If for some $\delta>0$, we set now 
$Z_{\mathbf{N},\delta}:=\{X\in M_n|\|[\mathbf{N},X]\|\leq \delta\}$, the set $Z_{\mathbf{N},\delta}$ is 
a matrix semialgebraic variety.
\end{example}

\begin{definition}[Curved and Flat matrix paths]
 Given any three hermitian matrices $-\mathbf{1}_n\leq H_1,H_2,H_3 \leq\mathbf{1}_n$ and a function $f\in C^1([-1,1])$, and given any four normal contractions 
 $D_1,\ldots,D_4$ in $M_n$, with $D_2=\Ad{e^{\pi i H_1}}{D_1}$, $D_3=f(H_2)$ and $D_4=f(H_3)$. Let us consider the paths $\breve{Z}_t:=\Ad{e^{\pi i t H_1}}{D_1}$ 
 and $\bar{V}_t:=f(tH_3+(1-t)H_2)$. We will say that $\breve{Z}$ is a {\em \bf curved} interpolating path for $D_1,D_2$ and we will say that the path 
 $\bar{V}$ is a {\em \bf flat} interpolating path for $D_3,D_4$.
\end{definition}

\begin{definition}[$\circledast$ operation]
 Given two matrix paths $X,Y\in C([0,1],M_n)$ we write ${X\circledast Y}$ to denote the concatenation of $X$ and $Y$, which is 
 the matrix path defined in terms of $X$ and $Y$ by the expression,
 \[
  {X\circledast Y}_s:=
  \left\{
  \begin{array}{l}
   X_{2s},\:\: 0\leq s\leq \frac{1}{2},\\
   Y_{2s-1},\:\: \frac{1}{2}\leq s\leq 1.
  \end{array}
  \right.  
 \]
\end{definition}

Let us denote by $\kappa$ the matrix compression $M_{2n}\to M_n$ defined by the mapping
 \begin{eqnarray}
  \kappa:M_{2n}\to M_n,
        \left(
        \begin{array}{cc}
         x_{11} & x_{12} \\
         x_{21} & x_{22}
        \end{array}
        \right)
        \mapsto x_{11}.    
 \label{the_compression}
 \end{eqnarray}
Let us write $\imath_2:M_n\to M_{2n}$ to denote the $C^*$-homomorphism defined by the expression $\imath_2(x):=x\oplus x=\I_2\otimes x$.

For any C$^*$-homomorphism $\Psi:M_n\to M_n$, we will write $\Psi^\dagger$ and $\Psi_{[m]}$ to denote the inverse of 
$\Psi$ and the natural extension 
$\Psi_{[m]}:M_n^m\to M_n^m,(X_1,\ldots,X_m)\mapsto(\Psi(X_1),\ldots,\Psi(X_m))$, respectively. 

\begin{definition}[Standard dilations]
\label{Std_dilations}
 Given a $C^*$-automorphism $\Psi:=\mathrm{Ad}[W]$ (with $W\in \U{n}$) in $M_n$, we will denote by 
$\Psi^{[s]}$ the $C^*$-automorphism in $M_{2n}$ defined by the expression 
$\Psi^{[s]}:=\mathrm{Ad}[\I_2\otimes W]=\mathrm{Ad}[W\oplus W]$. We call $\Psi^{[s]}$ a standard dilation 
of $\Psi$.
\end{definition}

\begin{definition}[$\ZZ/2$-dilations]
\label{Z2_dilations}
Given a $C^*$-automorphism $\Psi:=\mathrm{Ad}[W]$ (with $W\in \U{n}$) in $M_n$, we will denote by 
$\Psi^{[2]}$ the $C^*$-automorphism in $M_{2n}$ defined by the expression 
$\Psi^{[2]}:=\mathrm{Ad}[(\Sigma_2\otimes \I_n)(W^*\oplus W)]$. We call $\Psi^{[2]}$ a $\ZZ/2$-dilation 
of $\Psi$.
\end{definition}

\begin{remark}
\label{local_matching_remark}
 It can be seen that $\kappa(\imath_2(x))=x$ for any $x\in M_{2n}$, it can also be seen that 
 $\kappa((\Psi^\dagger)^{[2]}(\imath_2(x)))=\kappa(\Psi^{[s]}(\imath_2(x)))$.
\end{remark}

Using the same notation as Pryde in \cite{Pryde_Inequalities}, let $\RR_{(N)}$ denote the Clifford algebra over $\RR$ 
with generators $\NS{e}{N}$ and relations $e_ie_j=-e_je_i$ for $i\neq j$ and $e^2_i=-1$. Then $\RR_{(N)}$ is an 
associative algebra of dimension $2^N$. Let $S(N)$ denote the set $\mathscr{P}(\{1,\ldots,N\})$. Then the elements 
$e_S=e_{s_1}\cdots e_{s_k}$ form a basis when $S=\{\NS{s}{k}\}$ and $1\leq s_1< \cdots<s_k\leq N$. Elements of 
$\RR_{(N)}$ are denoted by $\lambda=\sum_{S}\lambda_Se_S$ where $\lambda_S\in \RR$. Under the inner product 
$\scalprod*{}{\lambda}{\mu}=\sum_S \lambda_S \mu_S$, $\RR_{(N)}$ becomes a Hilbert space with orthonormal basis $\{e_S\}$.

\begin{definition}
The {\em Clifford operator} of $N$ elements $\NS{X}{N}\in M_n$ is the operator 
defined in  $M_n\otimes \RR_{(N)}$   by
\[
 \Cliff{\NS{X}{N}}:=i\sum_{j=1}^N X_j\otimes e_j.
\]
\end{definition}
Each element $T=\sum_S T_S\otimes e_S\in M_n\otimes \RR_{(N)}$ acts on elements $x=\sum_S x_S\otimes e_S\in \CC^n\otimes \RR_{(N)}$ 
by $T(x):=\sum_{S,S'}T_s(x_{S'})\otimes e_Se_{S'}$. So $\Cliff{\NS{X}{N}}\in M_n\otimes \RR_{(N)}\subseteq \Li{\CC^n\otimes \RR_{(N)}}$. By 
$\|\Cliff{\NS{X}{N}}\|$ we will mean the operator norm of $\Cliff{\NS{X}{N}}$ as an element of $\Li{\CC^n\otimes \RR_{(N)}}$. As observed by 
Elsner in \cite[5.2]{Elsner_Perturbation_theorems} we have that
\begin{equation}
 \|\Cliff{\NS{X}{N}}\|\leq \sum_{j=1}^N\|X_j\|.
 \label{Clifford_bound}
\end{equation}

Let us denote by $\mathbf{\Delta}$ the function 
$\mathbf{\Delta}:M_n^m\times M_n^m\to \RR^+_0,(\mathbf{S},\mathbf{T})\mapsto \|\Cliff{\mathbf{S}-\mathbf{T}}\|$. As a consequence of the estimates obtained in 
\ref{Clifford_bound} we 
will have that $\dist{\mathbf{S}}{\mathbf{T}}\leq m\max_{1\leq j\leq m}\|S_j-T_j\|$.

\begin{definition}[Locally controlled functional calculus]\label{def_local_controllability}
 Given two integers $k,m\geq 1$ and a linear mapping $\Phi:M_{n}^m\to M_{kn}^m$, we say that a function 
 $f:M_{n}^m\to M_n^m$ is locally controlled by $\Phi$ if for any 
 $\mathbf{X},\mathbf{Y}\in \N{n}(\disk)^m$ we have that, 
 \begin{equation}
 \dist{\Phi(f(\mathbf{X}))}{\Phi(f(\mathbf{Y}))}\leq C_f\dist{f(\Phi(\mathbf{X}))}{f(\Phi(\mathbf{Y}))}
 \label{controllability_condition_1}
 \end{equation}
 for some constant $C_f$ which does not depend on $n$. Here $\mathbf{\Delta}$ is some suitable metric induced in $M_n^m$ by the operator norm.
\end{definition}

The following results were proved in \cite{Vides_homotopies}.

\begin{lemma}[Existence of isospectral approximants]
\label{Joint_spectral_variation_inequality_2}
 Given $\varepsilon>0$ there is $\delta> 0$ such that, for any $2$ families of $N$ pairwise commuting normal 
 matrices $\NS{x}{N}$ and $\NS{y}{N}$ 
 which satisfy the constraints $\|x_j-y_j\|\leq \delta$ for each $1\leq j\leq N$, there is a $C^*$-homomorphism  
 $\Psi$ such that $\sigma(\Psi(x_j))=\sigma(x_j)$, $[\Psi(x_j),y_j]=0$ and 
 $\max\{\|\Psi(x_j)-y_j\|,\|\Psi(x_j)-x_j\|\}\leq \varepsilon$, for each $1\leq j\leq N$.
\end{lemma}

\begin{theorem}[Local normal toral connectivity]
\label{local_connection_of_N_tuples_of_normal_contractions}
 Given $\varepsilon>0$ and any $n\in\ZZ^+$, there is $\delta>0$ such that, for any $2N$ normal contractions $\NS{x}{N}$ and $y_1,\ldots,$ $y_N$ in $M_n$ which 
satisfy the relations
\[
\left\{
\begin{array}{l}
 [x_j,x_k]=[y_j,y_k]=0, \:\:\: 1\leq j,k\leq N,\\
 \|x_j-y_j\|\leq \delta, \:\:\: 1\leq j\leq N,
\end{array}
\right.
\]
there exist $N$ toroidal matrix 
links $X^1,\ldots,X^N$ in $M_n$, which solve the problems
\[
 x_j \rightsquigarrow y_j, \:\:\: 1\leq j\leq N,
\]
and satisfy the constraints
\[
\left\{
\begin{array}{l}
 [X_t^j(x_j),X_t^k(x_k)]=0, \\
 \|X_t^j(x_j)-y_j\|\leq \varepsilon,
\end{array}
\right.
\]
for each $1\leq j,k\leq N$ and each $t\in \mathbb{I}$. Moreover, 
\[
 \ell_{\|\cdot\|}(X_t^j(x_j))\leq \varepsilon, \:\:\: 1\leq j\leq N.
\]
\end{theorem}

\begin{theorem}[Lifted local toral connectivity]
 \label{lifted_connectivity}
\label{lifted_local_connection_of_N_tuples_of_normal_contractions}
 Given $\varepsilon>0$, there is $\delta>0$ such that, for any $2N$ normal contractions $\NS{x}{N}$ and $\NS{y}{N}$ in $M_n$ which 
satisfy the relations
\[
\left\{
\begin{array}{l}
 [x_j,x_k]=[y_j,y_k]=0, \:\:\: 1\leq j,k\leq N,\\
 \|x_j-y_j\|\leq \delta, \:\:\: 1\leq j\leq N,
\end{array}
\right.
\]
there is a $C^*$-homomorphism $\Phi:M_n\to M_{2n}$ and $N$ toroidal matrix 
links $X^1,\ldots,X^N$ in $C(\mathbb{I},M_{2n})$, which solve the problems
\[
 \Phi(x_j) \rightsquigarrow y_j\oplus y_j, \:\:\: 1\leq j\leq N,
\]
and satisfy the constraints
\[
\left\{
\begin{array}{l}
 [X_t^j,X_t^k]=0, \\
 \kappa(\Phi(x_j))=x_j,\\
 \|\Phi(x_j)-x_j\oplus x_j\|\leq \varepsilon,\\
 \|X_t^j-y_j\oplus y_j\|\leq \varepsilon,
\end{array}
\right.
\]
for each $1\leq j,k\leq N$ and each $t\in \mathbb{I}$. Moreover, 
\[
 \ell_{\|\cdot\|}(X_t^j)\leq \varepsilon, \:\:\: 1\leq j\leq N.
\]
\end{theorem}

\subsection{Jointly compressible matrix sets}\label{compressible_matrix_sets} Given $0<\delta\leq \varepsilon$, we can now consider an alternative approach to 
the local connectivity problem involving two $N$-sets of pairwise commuting normal matrix contractions $\NS{X}{N}$ and $\NS{Y}{N}$ such 
that $\|X_j-Y_j\|\leq\delta$ for 
each $1\leq j\leq N$. The approach that we will consider in this section consists of considering the existence 
of a normal contraction $\hat{X}$ such that $\NS{X}{N}\in C^*(\hat{X})$, and which also satisfies the constraint $\|\hat{X}-X_j\|\leq \varepsilon$ for some 
$1\leq j\leq N$. A matrix $\hat{X}$ which satisfies the previous conditions will be called a {\em \bf nearby generator} for 
$\NS{X}{N}$, it can be seen that for any $\delta\leq \nu\leq \varepsilon$ one can find a flat analytic path $\bar{X}\in C([0,1],\Minf)$ 
that performs the deformation $X_j\rightsquigarrow_{\nu}\hat{X}$, where $\hat{X}$ is a nearby generator for 
$\NS{X}{N}$.

Given any joint isospectral approximant $\Psi$ with respect to the families normal contractions described in the previous paragraph, along the lines 
of the program that we have used 
to derive the connectivity results 
T.\ref{local_connection_of_N_tuples_of_normal_contractions} and T.\ref{lifted_local_connection_of_N_tuples_of_normal_contractions}, we 
can use L.\ref{Joint_spectral_variation_inequality_2} to find a C$^*$-automorphism which solves the extension problem described by 
the diagram,
\begin{equation}
\xymatrix{
& C^*(\hat{X}) \ar@{-->}[d]^{\hat{\Psi}}\\
C^*(\NS{X}{N}) \ar@{^{(}->}[ru] \ar[r]_{\Psi} & C^*(\NS{Y}{N})'
}
 \label{NJC_lifting}
\end{equation}
and satisfies the relations $\Psi(X_j)=\hat{\Psi}(X_j)$ for each $1\leq j\leq N$ together with the normed constraints
\[
\max\{\|\hat{\Psi}(\hat{X})-\hat{X}\|,\max_j\{\|\hat{\Psi}(X_j)-X_j\|,\|\hat{\Psi}(X_j)-Y_j\|\}\}\leq \varepsilon. 
\]
We refer to the C$^*$-automorphism $\hat{\Psi}$ in \ref{NJC_lifting} as a {\bf compression} of $\Psi$ or a {\bf compressive joint isospectral approximant} 
({\bf CJIA}) for the $N$-sets of normal contractions.

\section{Matrix Words and Structured Matrix Sets}
\label{matrix_words}

Given a finite set $\mathbf{C}:=\{\NS{c}{M}\}\subset M_n$ of normal contractions which contains the identity matrix $\I_n\in M_n$ 
and some fixed but arbitrary integer $L>0$, by a mixed matrix word of 
length $L$ we mean a function $W_L:M_n^{M}\times M_n^{2N}\to M_n,(\NS{c}{M},\NS{x}{{2N}})\mapsto c_{j_1}x_{j_1}^{k_1}\cdots c_{j_L}x_{j_L}^{k_L}$ on $2N$ 
matrix variables $\mathbf{X}:=\{x_{1},$
$\ldots,x_{2N}\}$, where $c_{j_l}\in \mathbf{C}$ and $k_l\in \mathbb{Z}^+_0$, $1\leq l\leq L$. The number 
$\deg(W_L):=\max_{1\leq l\leq L}\{k_l\}$ will be 
called the degree of the word $W_L$. We call the sets $\mathbf{C}$ and $\mathbf{X}$ matrix coefficient and matrix variable sets respectively. 

Let us consider now a function $F:M_n^{m}\to M_n^{m},\mathbf{X}\mapsto (f_1(\mathbf{X}),\ldots,f_m(\mathbf{X}))$ where for each $1\leq j\leq m$ we have that 
\begin{eqnarray}
f_k(\mathbf{X}):=\sum_{j=1}^{J_k} \alpha_{k,j} W_{1,L_j}(\mathbf{C},\NS{X}{m},\NS{X^*}{m}),
\label{matrix_function_on_words}
\end{eqnarray}
where $\{\alpha_{k,j}\}_{1\leq k\leq m,\\ 1\leq j\leq J_k}\subseteq \CC$ and $\mathbf{C}$ is a set of matrix coefficients of the corresponding matrix words. Building on 
the ideas presented in \cite[Chapt. 5]{Vides_dissertation}, let us consider the following concept.
%

\begin{remark}\label{controllability_remark}
Along the lines of the analysis presented in \cite[Chapt. 5]{Vides_dissertation}, it can be noticed that for any 
C$^*$-homomorphism $\Psi:M_n\to M_n$ the matrix functions of the form \ref{matrix_function_on_words} are locally 
controlled by $\Psi^{[s]}$ and $\Psi^{[2]}$.
\end{remark}

\subsection{Covering Matrix Varieties and Structured Matrix Algebras}

\label{covering_varieties}

\begin{definition}[Covering Matrix Semialgebraic Varieties]
 Given two integers $k,m\geq 1$, a matrix semialgebraic variety $\mathcal{Z}\subseteq M_n^m$ and a linear surjection 
 $\mathcal{K}: M_{kn}^m\to M_n^m$, we say that a matrix semialgebraic variety $\hat{\mathcal{Z}}\subseteq M_{kn}^m$ is a 
 {\em \bf covering (semialgebraic) variety} of $\mathcal{Z}$ if given $\varepsilon>0$, there is $\delta>0$ such that 
 for any $m$-tuple $\mathbf{X}\in \hat{\mathcal{Z}}$, $\mathcal{K}(\mathbf{X})\in \mathcal{Z}$ and if 
 $\dist{\mathcal{K}(\mathbf{X})}{\mathcal{K}(\mathbf{Y})}\leq \delta$ for some $\mathbf{X},\mathbf{Y}\in \hat{\mathcal{Z}}$, then 
 $\dist{\mathbf{X}}{\mathbf{Y}}\leq \varepsilon$.
\end{definition}

\begin{example}
 Let us consider the matrix $\mathbf{N}:=\diag{n,n-1,\ldots,1}$ in $M_n$, together with the corresponding matrix variety 
 \[
Z_{\mathbf{N}}:=\{(X_1,\ldots,X_m)\in \N{n}(\disk)^m|[\mathbf{N},X_j]=0, 1\leq j\leq m\}.
\]
For any C$^*$-homomorphism $\Psi:C^*(\mathbf{N})\to C^*(\mathbf{N})$. The matrix varieties 
$\hat{Z}_{\mathbf{N}}:=\Psi^{[s]}_{[m]}(Z_{\mathbf{N}})$ and 
$\tilde{Z}_{\mathbf{N}}:=\Psi^{[2]}_{[m]}(Z_{\mathbf{N}})$ are covering varieties of $Z_{\mathbf{N}}$ with respect to the natural extension to $M_n^m$ of the linear 
compression $\kappa$ defined in \ref{the_compression}.
\end{example}

\begin{definition}[Structured Matrix C$^*$-algebra]
 Given a matrix semialgebraic variety $\mathcal{Z}\subseteq \N{n}(\disk)^m$, we say 
 that the C$^*$-algebra $A_{\mathcal{Z}}:=C^*(X_1,\ldots,X_m)$ is a $\mathcal{Z}$-structured matrix (or just matrix structured when it is clear from the context) 
 C$^*$-algebra, if $\mathbf{X}:=(X_1,\ldots,X_m)\in \mathcal{Z}$.
\end{definition}

\begin{definition}[Structured C$^*$-homomorphisms]
Given a matrix semialgebraic variety $\mathcal{Z}\in M_n$ and a $\mathcal{Z}$-structured matrix C$^*$-algebra $A_{\mathcal{Z}}:=C^*(X_1,\ldots,X_m)\subseteq M_{n}$, we say that a C$^*$-homomorphism 
$\Psi:M_n \to M_n$ is $\mathcal{Z}$-structured if $\Psi(A_\mathcal{Z})\subseteq A_\mathcal{Z}$.
\end{definition}

\subsection{Local connectivity of covering matrix varieties}

\label{main_results}

Let us start by generalizing some concepts in \cite{Vides_homotopies}.

\begin{definition}[Generalized uniformly compressible JIA] Given $0<\delta\leq \varepsilon$ and two $N$-sets of pairwise commuting 
normal contractions $\NS{X}{N}$ and $\NS{Y}{N}$ in $\Minf$ such that $\|X_j-Y_j\|\leq \delta$, $1\leq j\leq N$, a joint isospectral approximant 
$\Psi$ of the $N$-sets is said to be {\bf uniformly compressible} if there are a nearby generator $\hat{X}$ for $\NS{X}{N}$, an extension $\hat{\Psi}:=\mathrm{Ad}[W]$ (with $W\in \U{\Minf}$) of $\Psi$ and a unitary 
$\hat{W}\in \hat{\Psi}(C^*(\hat{X}))'$ such that $\|W-\hat{W}\|\leq \varepsilon$. We refer to the $2N$ normal contractions 
$\NS{X}{N}$ and $\NS{Y}{N}$ for which there exists a uniformly compressible JIA ({\bf GUCJIA}) as uniformly jointly compressible ({\bf GUJC}).
\end{definition}

\begin{lemma}[Local connectivity of {\bf GUJC} matrix sets]
\label{GUJC_connectivity}
 Given $\varepsilon>0$, there is $\delta>0$, such that for any two $N$-sets of {\bf GUJC} pairwise commuting normal contractions 
 $\NS{X}{N}$ and $\NS{Y}{N}$ in $\Minf$ such that $\|X_j-Y_j\|\leq \delta$ for each $1\leq j\leq N$, we will have that there are 
 $N$ commutativity preserving piecewise analytic paths $\mathbf{X}^1,\ldots,\mathbf{X}^N\in C([0,1],\Minf)$ that solve the interpolation 
 problem $X_j \rightsquigarrow_{\varepsilon} Y_j$, for each $1\leq j\leq N$.
\end{lemma}
\begin{proof}
 Since the $N$-sets of pairwise commuting normal contractions $\NS{X}{N}$ and $\NS{Y}{N}$ are {\bf GUJC}, we have that given 
 $0<\delta\leq \nu\leq \varepsilon/2<1$, there are a normal contraction $\hat{X}\in \Minf$ which commutes with each $X_j$ together with a {\bf GUCJIA} 
 $\hat{\Psi}=\mathrm{Ad}[W]$ for some $W\in \U{\Minf}$ and a unitary $\hat{W}\in \hat{\Psi}(C^*(\hat{X}))'$ such that 
 \begin{equation}
  \|\I-\hat{W}^*W\|=\|W-\hat{W}\|\leq \nu<1.
  \label{compresion_inequality_1}
 \end{equation}
 Let us set $Z:=\hat{W}^*W$, as a consequence of the inequality \ref{compresion_inequality_1} we will have that there is a hermitian matrix $-\I\leq H_{Z}\leq \I$ in 
 $\Minf$ such that $e^{\pi iH_Z}=Z$. By using \ref{compresion_inequality_1} again, it can be seen that we can now use the curved paths 
 $\breve{\mathbf{X}}^j:=\Ad{e^{\pi i t H_Z}}{X_j}$ to solve the problems $X_j\rightsquigarrow_{\varepsilon/2} \hat{\Psi}(X_j)$, and then we can solve 
 the problems $\hat{\Psi}(X_j)\rightsquigarrow_{\nu} Y_j$ using the flat paths $\bar{\mathbf{X}}^j:=(1-t)\hat{\Psi}(X_j)+tY_j$. We can construct the solvent interpolating paths by setting $\mathbf{X}^j:=\breve{\mathbf{X}}^j\circledast \bar{\mathbf{X}}^j$ for each $1\leq j\leq N$. This completes the proof.
\end{proof}

\subsubsection{Local connectivity of algebraic contractions}
\label{local_connectivity_results}
Given $p,q\in \CC[z]$ with $\mathbf{Z}(p)\subseteq \TT$ and $\mathbf{Z}(q)\subseteq \disk$. Let us consider the universal C$^*$-algebras 
$\mathbb{U}[p]:=C^*_1\langle u| uu^*=u^*u=1, \: p(u)=0\rangle$ and $\mathbb{D}[q]:=C^*_1\langle z| zz^*=z^*z, \: \|z\|\leq 1, \: q(z)=0\rangle$.

\begin{definition}[Algebraic contractions]
For fixed but arbitrary $p\in \CC[z]$ with $\mathbf{Z}(p)\subseteq \disk$, by 
a $p$-algebraic (or just algebraic when $p$ is clear from the context) contraction in $M_n$, we mean any element $\mathbf{D}\in \U{n}$ such that the mapping 
$z\mapsto \mathbf{D}$ induces a C$^*$-homomorphism 
$\mathbb{D}[p]\to C^*(\mathbf{D})$, where $z$ denotes the universal generator of $\mathbb{D}[p]$. If in addition, 
$\mathbf{Z}(p)\subseteq \TT$, $\mathbf{D}\mathbf{D}^*=\I_n$ and $u\mapsto \mathbf{D}$ extends to a C$^*$-homomorphism 
$\mathbb{U}[p]\to C^*(\mathbf{D})$, we call any such $\mathbf{D}$ an algebraic unitary.
\end{definition}

\begin{lemma}
 \label{existence_of_a refinement}
 Given a unital abelian C$^*$-algebra $D$ and any two finite families of pairwise orthogonal projections 
 $\mathcal{P}:=\{P_1,\ldots,P_r\}$ and $\mathcal{Q}:=\{Q_1,\ldots,Q_s\}$ in $D$ such that 
 $1_D=\sum_{j}P_j=\sum_kQ_k$, there is a family  
 $\mathcal{R}:=\{R_1,\ldots,R_t\}$ of pairwise orthogonal projections in $D$ such that 
 $\mathrm{span}\: \{\mathcal{P},\mathcal{Q}\}\subseteq \mathrm{span}\: {\mathcal{R}}$ and 
 $|\mathcal{R}|\leq \|\mathcal{P}||\mathcal{Q}|$.
\end{lemma}
\begin{proof}
 Since $\mathcal{P},\mathcal{Q}\subset D$, by setting $R_{j,k}:=P_jQ_k$ it can be seen that 
 $\mathcal{P},\mathcal{Q}\subseteq \mathrm{span} \:\{R_{j,k}\}$. Let us set $\mathcal{R}:=\{R_{j,k}\}$, it can 
 be seen that $|\mathcal{R}|\leq |\mathcal{P}||\mathcal{Q}|$ and 
 $\mathrm{span}\: \{\mathcal{P},\mathcal{Q}\}\subseteq \mathrm{span}\: {\mathcal{R}}$. This completes the proof.
\end{proof}

Along the lines of the proof of \cite[P.VI.6.6]{Bhatia_mat_book} we can derive the following lemma.

\begin{lemma}
\label{existence_of_almost_unit}
 Given a unitary $W$ and a normal contraction $D$ in $M_n$, if $D=\sum_{j=1}^r \alpha_jP_j$ is diagonal for $1\leq r\in \ZZ$ and $\NS{\alpha}{r}\in \disk$, the 
 set $\{P_j\}$ consists of pairwise orthogonal diagonal projections in $M_n$ such that $\sum_{j}P_j=\I_n$, and $\alpha_j\neq \alpha_k$ whenever $k\neq j$, then there 
 is a unitary matrix $Z\in M_n$ and a constant $C$ depending on $r$ and $\sigma(D)$ such that $[Z,D]=0$ and $\|\I_n-WZ\|\leq C\|WDW^*-D\|$.
\end{lemma}
\begin{proof}
 Since there are $r$ mutually orthogonal 
 projections $\mathbf{0}_n\leq P_1,\ldots,P_r\leq \I_n$ in $M_n$ such that $\sum_{j}P_j=\I_n$ and 
 $D:=\sum_{j}\alpha_jP_j$ with $\alpha_j\in \disk$. By setting 
 $W_{j,k}:=P_jWP_k$, we will have that $W$ has a decomposition $W=\sum_{j,k}W_{j,k}$ and it can be seen that 
\begin{eqnarray}
 \|WDW^*-D\|&=&\|WD-DW\|\\
            &=&\|\sum_{j,k}(\alpha_j P_jW_{j,k}-\alpha_k W_{j,k}P_k\|\\
            &=&\|\sum_{j,k}(\alpha_j-\alpha_k)W_{j,k}\|.
\end{eqnarray}
Hence, for $j\neq k$,
\begin{eqnarray}
 \|W_{j,k}\|&\leq& \frac{1}{|\alpha_j-\alpha_k|}\|WDW^*-D\|\\
            &\leq& \max_{j,k}\left\{\frac{1}{|\alpha_j-\alpha_k|}\right\}\|WDW^*-D\|.
\end{eqnarray}
Hence, by setting $s=\min_{j,k}|\alpha_j-\alpha_k|$ we will have that
\[
 \left\|W-\sum_{j}W_{j,j}\right\|\leq \frac{r(r-1)}{s}\|WDW^*-D\|.
\]
Let $X:=\sum_{j}W_{j,j}=\sum_{j}P_jWP_j$. Hence $\|X\|\leq \|W\|=1$. Let 
$W_{j,j}:=V_jR_j$ be the polar decomposition of $W_{j,j}$, with $V_jV_j^*=V_j^*V_j=P_j$, $R_j\geq \mathbf{0}_n$ and 
$[P_j,V_j]=[P_j,R_j]=0$. Then
\[
 \|W_{j,j}-V_j\|=\|R_j-P_j\|\leq \|R_j^2-P_j\|,
\]
since $R_j$ is a contraction. Let $V:=\sum_{j}V_j$. Then $V\in \U{n}$ and from the above inequality, we see that 
\[
 \|X-V\|\leq \|X^*X-\I_n\|=\|X^*X-W^*W\|.
\]
Hence,
\begin{eqnarray}
\|V-W\|&\leq&\|V-X\|+\|X-W\|\leq\|W-X\|+\|X^*X-W^*W\|\\
          &\leq&\|W-X\|+\|(X^*-W^*)X\|+\|W^*(X-W)\|\\
          &\leq&3\|W-X\|\leq \frac{3r(r-1)}{s}\|WDW^*-D\|.       
\end{eqnarray}
By setting $Z:=V^*$ and $C:=\frac{3r(r-1)}{s}$, it can be seen that $\|\I_n-WZ\|=\|V-W\|\leq C\|WDW^*-D\|$ and 
also that $[Z,D]=[V,D]=0$. This completes the proof.
\end{proof}

\begin{remark}\label{finite_order_remark}
 Given any normal contraction $D$ such that $p(D)=\mathbf{0}_n$ for some $p\in \CC[z]$ with $\deg(p)\leq r$, we will have that there are at most 
 $r$ complex numbers $\NS{\alpha}{r}\in \disk$ 
 and $r$ pairwise orthogonal projections 
 $P_1,\ldots,P_r$ such that $p(\alpha_j)=0$, $\sum_jP_j=\I$ and $U=\sum_{j}\alpha_jP_j$. 
\end{remark}

\begin{lemma}[Local algebraic contractive connectivity]
\label{local_connectivity_N_tuples_commuting_contractions}
 Given any $\varepsilon\geq 0$ and $N$ polynomials $\NS{p}{N}\in \CC[z]$, there is $\delta\geq 0$ such that for 
 any integer $n\geq 1$ and any $2N$ normal contractions $X_{1},\ldots,X_{N}$ ,$Y_{1},\ldots,Y_{N}$ in $M_{n}$ which satisfy the relations 
  \begin{eqnarray*}
  \left\{
  \begin{array}{l}
   [X_{j},X_{k}]=[Y_{j},Y_{k}]=0,\\
   p_j(X_{j})=p_j(Y_j)=\mathbf{0}_n,\\
   \|X_{k}-Y_{k}\|\leq \delta,
  \end{array}
 \right.
 \end{eqnarray*}
for each $1\leq j,k \leq N$, there are $N$ local analytic matrix homotopies 
$Z^1,\ldots,Z^N$ in $M_{n}$ which solve the interpolation problems 
 \[
 X_{k}\rightsquigarrow Y_{k}, \:\:\: 1\leq k\leq N,
 \]
 and also satisfy the relations 
 \begin{eqnarray*}
  \left\{
  \begin{array}{l}
     [Z^j_{t},Z^k_{t}]=\mathbf{0}_n,\\
   p_j(Z^j_t)=\mathbf{0}_n,\\
   (Z^j_{t})^*Z^j_{t}=Z^j_{t}(Z^j_{t})^*,\\
   \|Z^j_t-Y_j\|\leq \varepsilon,
  \end{array}
 \right.
 \end{eqnarray*}
for each $t\in\mathbb{I}$ and each $1\leq j,k\leq N$.
\end{lemma}
\begin{proof}
 By changing basis if necessary, we can assume that $\NS{Y}{N}$ are diagonal matrices. Let us set 
 $K:=1+\prod_{j=1}^N \max\{1,\deg(p_j)\}$ and let us consider the sets $\mathbf{Z}(p_j)=\{z\in \disk| p_j(z)=0\}$, $1\leq j\leq N$. 
 Given $\varepsilon>0$, there is 
 $\delta>0$ that can be chosen so that
 \begin{equation}
 \delta\leq \frac{h_\delta\varepsilon}{6K(K-1)}< \min\{h_\delta,\frac{1}{2}\} \leq \frac{1}{3}\min_{1\leq j\leq N}\{\min_{x,y\in \mathbf{Z}(p_j)}\{|x-y|\:|\:x\neq y\}\}
 \end{equation}
 with $h_\delta\leq \frac{1}{3}\min_{1\leq j\leq N}\{\min_{x,y\in \mathbf{Z}(p_j)}\{|x-y|\:|\:x\neq y\}\}$. By \cite[L.4.1]{Vides_homotopies} we will have that there 
 is an inner C$^*$-automorphism $\Psi:M_n\to M_n$ such that $[\Psi(X_j),Y_j]=0$ and $\|\Psi(X_j)-Y_j\|\leq \delta$, since 
 $\delta<\frac{1}{3}\min_{1\leq j\leq N}\{\min_{x,y\in \mathbf{Z}(p_j)}\{|x-y|\:|\:x\neq y\}\}$ and 
 $p_j(\Psi(X_j))=p_j(Y_j)=\mathbf{0}_n$ we will have that there is a unitary $\hat{W}\in \U{n}$ such that 
 $Y_j=\hat{W}X_j\hat{W}^*=\Psi(X_j)$, otherwise we get a contradiction. 
 
 By R.\ref{finite_order_remark} and by iterating on L.\ref{existence_of_a refinement} we can find a family of pairwise orthogonal projections 
 $\mathcal{P}:=\{P_1,\ldots,P_K\}$ with $\sum_{j}P_j=\I_n$ and $Y_j\in \mathrm{span}\: \mathcal{P}$ for each $1\leq j\leq N$. 
 
 By choosing any arbitrary $1\leq j\leq N$, we can now find a contractive perturbation $\hat{X}$ of $X_j$ such that 
 $\Psi(\hat{X}):=\sum_{j=1}^K \alpha_jP_j \in \mathrm{span}\:\mathcal{P}$, $\|\hat{X}-X_j\|\leq \delta/2$, 
 $\sigma(\hat{X})$ consists of $K$ distinct points and such that $\sigma(X_j)$ is $\delta/2$-dense in $\sigma(\hat{X})$. By the hypotheses of the 
 lemma, we can now obtain 
 the estimates
 \begin{eqnarray}
  \|\hat{W}\hat{X}-\hat{X}\hat{W}\|&=&\|\hat{W}(\hat{X}-X_j)-(\hat{X}-X_j)\hat{W}+\hat{W}X_j-X_j\hat{W}\|\\
                       &\leq&2\|\hat{X}-X_j\|+\|\hat{W}X_j-X_j\hat{W}\|\\
                       &\leq&2\frac{\delta}{2}+\|X_j-Y_j\|\leq 2\delta.
 \end{eqnarray}

 By L.\ref{existence_of_almost_unit} there is a 
 unitary $Z$ such that $[Z,\Psi(\hat{X})]=0$ and $\|Z-\hat{W}\|\leq \varepsilon$, and which also satisfies the relations 
 $[Z,\Psi(X_j)]=[Z,Y_j]=0$, $1\leq j\leq N$, if we set $W:=\hat{W}^*Z$, we will have that 
 $WY_jW^*=\hat{W}^*Y_j\hat{W}=\Psi^{-1}(Y_j)=X_j$, $1\leq j\leq N$. Moreover, by R.\ref{finite_order_remark} and 
 L.\ref{existence_of_almost_unit} we will have that the normal contractions 
 $\NS{X}{N}$ and $\NS{Y}{N}$ have the {\bf GUJC} condition and we can apply L.\ref{GUJC_connectivity}
to construct the paths $Z^j_t\in C([0,1],\N{n})$ that solve the interpolation problems 
$X_j \rightsquigarrow_{\varepsilon} Y_j$ together with the rest of the constraints in the statement of the lemma, and we 
are done.
\end{proof}

\subsubsection{Soft Algebraic Contractions} 
Along the lines followed in \cite{Vides_homotopies} let us now consider two particular types of matrix paths defined 
as follows.

 \begin{lemma}[Local Soft algebraic contractive connectivity]
\label{local_connectivity_N_tuples_soft_commuting_contractions}
 Given any $\varepsilon\geq 0$ and $N$ polynomials $\NS{p}{N}\in \CC[z]$, there is $\delta\geq 0$ such that for 
 any integer $n\geq 1$ and any $2N$ normal contractions $X_{1},\ldots,X_{N}$ ,$Y_{1},\ldots,Y_{N}$ in $M_{n}$ which satisfy the relations 
  \begin{eqnarray*}
  \left\{
  \begin{array}{l}
   [X_{j},Y_{k}]=[Y_{j},Y_{k}]=0,\\
   \max\{\|p_j(X_{j})\|,\|p_j(Y_j)\|\}\leq \delta,\\
   \|X_{k}-Y_{k}\|\leq \delta,
  \end{array}
 \right.
 \end{eqnarray*}
for each $1\leq j,k \leq N$, there are $N$ local piecewise analytic matrix homotopies 
$Z^1,\ldots,Z^N$ in $M_{n}$ which solve the interpolation problems 
 \[
 X_{k}\rightsquigarrow Y_{k}, \:\:\: 1\leq k\leq N,
 \]
 and also satisfy the relations 
 \begin{eqnarray*}
  \left\{
  \begin{array}{l}
     [Z^j_{t},Z^k_{t}]=\mathbf{0}_n,\\
   \|p_j(Z^j_t)\|\leq\varepsilon,\\
   (Z^j_{t})^*Z^j_{t}=Z^j_{t}(Z^j_{t})^*,\\
   \|Z^j_t-Y_j\|\leq \varepsilon,
  \end{array}
 \right.
 \end{eqnarray*}
for each $t\in\mathbb{I}$ and each $1\leq j,k\leq N$.
\end{lemma}
\begin{proof}
 By a similar perturbation argument to the one implemented in the proof of L.\ref{local_connectivity_N_tuples_commuting_contractions} and 
 by matrix Lipschitz continuity of polynomials we can find $0<\delta\leq \varepsilon/6$ such that 
 \begin{equation}
 \delta\leq \frac{1}{6}\min_{1\leq j\leq N}\{\min_{x,y\in \mathbf{Z}(p_j)}\{|x-y|\:|\:x\neq y\}\}.
 \label{spectral_constraint_1}
 \end{equation}
 and such that $\mathbf{Z}(p_j)$ is $\delta/2$-dense in $\sigma(X_j)\cup \sigma(Y_j)$. Along the lines of the proof of \cite[C.4.2]{Vides_homotopies}, we can find two flat 
 paths $\bar{X}^j,\bar{Y}^j\in C([0,1],M_{n})$ that preserve normality, commutativity and solve the problems $X_j \rightsquigarrow_{\delta/2} \hat{X}_j$ and 
 $\hat{Y}_j \rightsquigarrow_{\delta/2} Y_j$ respectively, with $p_j(\hat{X}_j)=p_j(\hat{Y}_j)=[\hat{X}_j,X_j]=[\hat{Y}_j,Y_j]=\mathbf{0}_n$. We can now use the estimate 
 $\|\hat{X}_j-\hat{Y}_j\|\leq \delta/2+\delta+\delta/2=2\delta$ together with inequality \ref{spectral_constraint_1} and 
 L.\ref{local_connectivity_N_tuples_commuting_contractions} to construct the curved paths $\breve{Z}^j\in C([0,1],\N{n})$ that solve 
 the problems $\hat{X}_j \rightsquigarrow_{\varepsilon/3} \hat{Y}_j$ for each $1\leq j\leq N$. It can be seen that the paths 
 $Z^j:=(\bar{X}^j\circledast \breve{Z}^j) \circledast \bar{Y}^j$ solve the problems $X_j\rightsquigarrow_\varepsilon Y_j$ and satisfy the constraints 
 in the statement of the lemma. This completes the proof.
\end{proof}

\begin{remark}
 The application/modification of L.\ref{local_connectivity_N_tuples_soft_commuting_contractions} to the solution of unitary connectivity problems together with some connectivity results for matrix representations of softened abelian group 
 C$^*$-algebras in the sense of Farsi \cite{Soft_Algebras_Farsi} raises interesting questions, this will be the subject of further 
 study. We can also combine L.\ref{local_connectivity_N_tuples_commuting_contractions} 
 and L.\ref{local_connectivity_N_tuples_soft_commuting_contractions} with 
 the techniques presented in \cite{Vides_homotopies} to derive some results in matrix numerical analysis and topologically controlled linear algebra. This will be the subject of 
 further study
\end{remark}

\subsection{Local Deformation of Matrix Words} 
\label{Applications}
Let us start this section by considering the following concepts.

 \begin{definition}[Local deformation of Structured Matrix C$^*$-algebras]
 Given a matrix variety $\mathcal{Z}\subseteq \Minf$, two $m$-tuples of structured normal contractions 
 $(X_,\ldots,\\
 X_m)$ and $(\NS{Y}{m})$ in $M_n^m\cap \mathcal{Z}^m$ we write $C^*(\NS{X}{m})\rightsquigarrow_{\varepsilon} C^*(\NS{Y}{m})$ if there are $\varepsilon$-local matrix 
 homotopies 
 $Z^1,\ldots,Z^m\in C([0,1],\mathcal{Z}\cap M_n)$ between $\NS{X}{m}$ and $\NS{Y}{m}$.
\end{definition}


\begin{definition}[ULPAC.] \label{def.ulac}Given a universal C$^*$-algebra $A:=C^*_1\langle\NS{x}{m},\mathcal{R}(x_1,\\
,\ldots,x_m)\rangle$, we say that 
$\Rep{A}{\Minf}$ is uniformly locally piecewise analytically connected or 
({\bf ULPAC}) if for 
any two $m$-tuples of normal contractions $(\NS{X}{m})$ and $(Y_1,\ldots,Y_m)$ in 
 $M_n^m$ we have that if $C^*(\NS{X}{m}) \twoheadleftarrow A\twoheadrightarrow C^*(Y_1,\ldots,Y_m)$, then we have that
 $X_j \rightsquigarrow_{\varepsilon(\delta),\hat{X}^j} Y_j$, for some function $\varepsilon:\RR^+_0\to \mathbb{R}_0^+$ 
 and some piecewise analytic contractive normal matrix path $\hat{X}^j$, with 
 $\delta:=\max_j\{\|X_j-Y_j\|\}$. Moreover, $A\twoheadrightarrow C^*(\hat{X}_{t}^1,\ldots,\hat{X}_{t}^N)$, for each 
 $0\leq t\leq 1$.
\end{definition}

\begin{definition}[AULPAC.] Given a universal C$^*$-algebra $A:=C^*_1\langle\NS{x}{m},\\
\mathcal{R}(x_1,\ldots,x_m)\rangle$, we say that 
$\Rep{A}{\Minf}$ is approximately uniformly locally piecewise analytically connected ({\bf AULPAC}) if for 
any two $m$-tuples of normal contractions $(\NS{X}{m})$ and $(\NS{Y}{m})$ in 
 $M_n^m$ we have that if $C^*(\NS{X}{m}) \twoheadleftarrow A\twoheadrightarrow C^*(\NS{Y}{m})$, then there is 
 a C$^*$-homomorphism $\Phi:M_n\to M_{kn}$, $k\geq 1$ such that $\max\{\|\Phi(X_j)-\I_k\otimes X_j\|,\|\Phi(X_j)-\I_k\otimes Y_j\|\}\leq \varepsilon(\delta)$ and 
 $\Phi(X_j) \rightsquigarrow_{\varepsilon(\delta),\hat{X}^j} \I_k\otimes Y_j$, for some function $\varepsilon:\RR^+_0\to \mathbb{R}_0^+$ and some piecewise analytic contractive normal matrix path $\hat{X}^j$, with 
 $\delta:=\max_j\{\|X_j-Y_j\|\}$. Moreover, $A\twoheadrightarrow C^*(\hat{X}_{t}^1,\ldots,\hat{X}_{t}^N)$, for each 
 $0\leq t\leq 1$.
\end{definition}

Let us consider the universal C$^*$-algebras $\mathbb{I}^{m}$ and $\mathbb{S}_\varepsilon^{m-1}$ described in term of generators and relations by the expressions.
\begin{eqnarray}
\mathbb{I}^{m}:=
 C^*_1\left\langle h_1,\ldots,h_m \left| 
                     \begin{array}{l}
                       -1\leq h_j\leq 1, \: [h_j,h_k]=0
                      \end{array}, 1\leq j,k\leq m
\right.\right\rangle,\nonumber\\
\label{universal_picture_of_Bm}
\\
\mathbb{S}^{m-1}_\varepsilon:=
 C^*_1\left\langle h_1,\ldots,h_m \left| 
                     \begin{array}{l}
                       -1\leq h_j\leq 1, \:[h_j,h_k]=0 \\
                       \|h_1^2+\cdots+h_m^2-1\|\leq \varepsilon
                      \end{array}, 1\leq j,k\leq m
\right.\right\rangle.\nonumber\\
\label{universal_picture_of_Sm}
\end{eqnarray}

Given $\varepsilon>0$ and a $m$-set of polynomials 
$p_1,\ldots,p_{m}\in \CC[z]$ (with $m\geq 2$) let us denote by $\mathbb{I}^{m}_\varepsilon[p_1,\ldots,p_{m}]$ and $\mathbb{S}^{m-1}_\varepsilon[p_1,\ldots,p_{m}]$ 
the universal C$^*$-algebras defined by
\begin{eqnarray}
\mathbb{I}^{m}_\varepsilon[p_1,\ldots,p_{m}]:=
 C^*_1\left\langle h_1,\ldots,h_m \left| 
                     \begin{array}{l}
                       -1\leq h_j\leq 1, \:[h_j,h_k]=0 \\
                       \|p_j(h_j)\|\leq\varepsilon
                      \end{array}, 1\leq j,k\leq m
\right.\right\rangle,\nonumber\\
\label{universal_picture_of_alg_Bm}\\
\mathbb{S}^{m-1}_\varepsilon[p_1,\ldots,p_{m}]:=
 C^*_1\left\langle h_1,\ldots,h_{m} \left| 
                     \begin{array}{l}
                       -1\leq h_j\leq 1, \:[h_j,h_k]=0 \\
                       \|h_1^2+\cdots+h_{m}^2-1\|\leq \varepsilon\\
                       \|p_j(h_j)\|\leq \varepsilon
                      \end{array}, 1\leq j,k\leq m
\right.\right\rangle.\nonumber\\
\label{universal_picture_of_alg_Sm}
\end{eqnarray}
We will refer to $\mathbb{I}^{m}$ and $\mathbb{I}^{m}_\varepsilon[p_1,\ldots,p_{m}]$ as {\bf Semi-Soft m-Cubes}, we will refer 
$\mathbb{S}^{m-1}_\varepsilon$ and $\mathbb{S}^{m-1}_\varepsilon[p_1,\ldots,p_{m}]$ as {\bf Semi-Soft m-Spheres}.

Since by R.\ref{controllability_remark} any function on matrix words of the form \ref{matrix_function_on_words} safisfies the controllability conditions of 
D.\ref{def_local_controllability}, we will have that the {\bf curved} local {\bf connectivity/deformation} of matrix words on normal variables can be reduced 
to the following results.

\begin{theorem}
 For any integer $m\geq 2$ we have that $\Rep{\mathbb{I}^{m}_\varepsilon[p_1,\ldots,p_m]}{\Minf}$ is {\bf ULPAC}.
\end{theorem}
\begin{proof}
Since we have that $\mathbb{I}^{m}_\varepsilon[p_1,\ldots,p_m]$ is described by \ref{universal_picture_of_alg_Bm}. It can be seen that 
 $\|p_j(\rho_n(h_j))\|\leq \varepsilon$ for $\rho_n\in \Rep{\mathbb{I}^{m}_\varepsilon[p_1,\ldots,p_m]}{M_n}$ and each $1\leq j\leq m$. 
 By L.\ref{local_connectivity_N_tuples_soft_commuting_contractions}, for any given $\varepsilon>0$, there is $\delta> 0$ such that 
 for any two representations $M_n\supseteq C^*(\NS{H}{m}) \twoheadleftarrow \mathbb{I}^{m}_\varepsilon[p_1,\ldots,p_{m}] 
 \twoheadrightarrow C^*(\NS{K}{m})\subseteq M_n$, with $\|H_j-K_j\|\leq \delta$ for each $1\leq j\leq m$, there are 
 $m$ hermitian contractive paths $Z_{j}\in C([0,1],\mathbb{H}(n))$ that solve the problems 
 $H_j \rightsquigarrow_{\varepsilon} K_j$, such that $\mathbb{I}^{m}_\varepsilon[p_1,\ldots,p_m] 
 \twoheadrightarrow C^*(Z_{1,t},\ldots,Z_{m,t})\subseteq M_n$ for each $0\leq t\leq 1$. This completes the proof.
\end{proof}

\begin{theorem}
 For any integer $m\geq 2$ we have that $\Rep{\mathbb{S}^{m-1}_\varepsilon[p_1,\ldots,p_m]}{\Minf}$ is {\bf ULPAC}.
\end{theorem}
\begin{proof}
Since we have that $\mathbb{S}^{m-1}_\varepsilon[p_1,\ldots,p_m]$ is described by \ref{universal_picture_of_alg_Sm}. It can be seen that 
 $\|p_j(\rho_n(h_j))\|\leq \varepsilon$ for $\rho_n\in \Rep{\mathbb{S}^{m-1}_\varepsilon[p_1,\ldots,p_m]}{M_n}$ and each $1\leq j\leq m$. 
 By L.\ref{local_connectivity_N_tuples_soft_commuting_contractions}, for any given $\varepsilon>0$, there is $\delta> 0$ such that 
 for any two representations $M_n\supseteq C^*(\NS{H}{m}) \twoheadleftarrow \mathbb{S}^{m-1}_\varepsilon[p_1,\ldots,p_{m}] 
 \twoheadrightarrow C^*(\NS{K}{m})\subseteq M_n$, with $\|H_j-K_j\|\leq \delta$ for each $1\leq j\leq m$, there are 
 $m$ hermitian contractive paths $Z_{j}\in C([0,1],\mathbb{H}(n))$ that solve the problems 
 $H_j \rightsquigarrow_{\varepsilon} K_j$, such that $\mathbb{S}^{m-1}_\varepsilon[p_1,\ldots,p_m] 
 \twoheadrightarrow C^*(Z_{1,t},\ldots,Z_{m,t})\subseteq M_n$ for each $0\leq t\leq 1$. This completes the proof.
\end{proof}

\begin{theorem}
 For any integer $m\geq 2$ we have that $\Rep{\mathbb{I}^{m}}{\Minf}$ is {\bf AULPAC}.
\end{theorem}
\begin{proof}
Since we have that $\mathbb{I}^{m}$ is described by \ref{universal_picture_of_Bm}. By \cite[T.4.2]{Vides_homotopies}, given 
$\varepsilon>0$, there is $\delta> 0$ such that for any two representations $M_n\supseteq C^*(\NS{H}{m}) \twoheadleftarrow \mathbb{I}^{m} 
 \twoheadrightarrow C^*(\NS{K}{m})\subseteq M_n$, with $\|H_j-K_j\|\leq \delta$ for each $1\leq j\leq m$, there are 
a C$^*$-homomorphism $\Phi:M_n\to M_{2n}$ with $\max\{\|\Phi(H_j)-\I_2\otimes H_j\|,\|\Phi(H_j)-\I_2\otimes K_j\|\}\leq \varepsilon(\delta)$ and 
 $m$ contractive hermitian paths $W_{j}\in C([0,1],\mathbb{H}(2n))$ that solve the problems 
 $\Phi(H_j) \rightsquigarrow_{\varepsilon} \I_2\otimes K_j$, such that $\mathbb{I}^{m} 
 \twoheadrightarrow C^*(W_{1,t},\ldots,W_{m,t})\subseteq M_{2n}$ for each $0\leq t\leq 1$. This completes the proof.
\end{proof}

\begin{theorem}
 For any integer $m\geq 2$ we have that $\Rep{\mathbb{S}^{m-1}_\varepsilon}{\Minf}$ is {\bf AULPAC}.
\end{theorem}
\begin{proof}
Since we have that $\mathbb{S}^{m-1}_\varepsilon$ is described by \ref{universal_picture_of_Sm}. By \cite[T.4.2]{Vides_homotopies}, given 
$\varepsilon>0$, there is $\delta> 0$ such that for any two representations $M_n\supseteq C^*(\NS{H}{m}) \twoheadleftarrow \mathbb{S}^{m-1}_\varepsilon 
 \twoheadrightarrow C^*(\NS{K}{m})\subseteq M_n$, with $\|H_j-K_j\|\leq \delta$ for each $1\leq j\leq m$, there are 
a C$^*$-homomorphism $\Phi:M_n\to M_{2n}$ with $\max\{\|\Phi(H_j)-\I_2\otimes H_j\|,\|\Phi(H_j)-\I_2\otimes K_j\|\}\leq \varepsilon(\delta)$ and 
 $m$ contractive hermitian paths $W_{j}\in C([0,1],\mathbb{H}(2n))$ that solve the problems 
 $\Phi(H_j) \rightsquigarrow_{\varepsilon} \I_2\otimes K_j$, such that $\mathbb{S}^{m-1}_\varepsilon 
 \twoheadrightarrow C^*(W_{1,t},\ldots,W_{m,t})\subseteq M_{2n}$ for each $0\leq t\leq 1$. This completes the proof.
\end{proof}

\section{Hints and Future Directions}

The detection of nearby matrix representations of Semi-Soft Cubes and Spheres that can be uniformly deformed locally, provides 
interesting connections with some problems and questions in topologically controlled linear algebra raised by M. H. Freedman and 
R. Kirby, these connections togther with some related computational procedures will be studied in \cite{soft_computations}.

The geometric ideas and concepts developed by M. A. Rieffel in \cite{Vector_bundles_Rieffel} together with the techniques introduced by T. A. Loring and G. K. Pedersen in 
\cite{Smoothing_techniques} and by D. Hadwin in \cite{Lift_alg_elem_Don}, can be combined with the deformation/connectivity techniques presented in this document 
to obtain some information on the local geometric structure of $\Rep{C_\varepsilon(\mathbb{S}^2)}{\Minf}$ and 
$\Rep{C_\varepsilon([-1,1]^2\times [-1,1]^2)}{\Minf}$, these results will be presented in future communications. 

The curved nature of $\Rep{C_\varepsilon([-1,1]^2\times [-1,1]^2)}{\Minf}$ is 
particularly intersting because of its implicit relation with some equivalent formulations of Connes's embedding problem, and will be 
the subject of further study.

Some applications 
to the analysis of molecular data and related processes in chemical engineering and Graphene Nanotechnology (in the sense of \cite{NanoMesh1} and \cite{NanoMesh2}) 
are also considered as a subject of further study.

\section{Acknowledgement}
I am grateful with Terry Loring, Alexandru Chirvasitu, Moody Chu and Concepci\'on Ferrufino, for several interesting questions and 
comments that have been very helpful for the preparation of this document.

\end{document}